\documentclass{amsart}
\usepackage{amsmath,amssymb}

\newtheorem{theorem}{Theorem}

\newtheorem{lemma}[theorem]{Lemma}
\newtheorem{corollary}[theorem]{Corollary}

\theoremstyle{definition}

\theoremstyle{remark}

\newcounter{abschnitt}
\newcounter{saveeqn}
\newcommand{\alpheqn}{\setcounter{saveeqn}{\value{abschnitt}}
\renewcommand{\theequation}{\mbox{\arabic{saveeqn}.\arabic{equation}}}}
\newcommand{\reseteqn}{\setcounter{equation}{0}
\renewcommand{\theequation}{\arabic{equation}}}

\begin{document}

\title[An Asymmetric Affine P\'olya--Szeg\"o Principle]{An Asymmetric Affine P\'olya--Szeg\"o Principle}


\author{Christoph Haberl}

\address{Department of Mathematics, Polytechnic Institute of NYU\\
              Brooklyn, New York, USA} \email{chaberl@poly.edu}
\author{Franz E. Schuster}
\address{Institute of Discrete Mathematics and Geometry, Vienna University of
Technology} \email{franz.schuster@tuwien.ac.at}

\thanks{Jie Xiao was in part supported by the Natural Science and
Engineering Research Council of Canada.}



\author{Jie Xiao}
\address{Department of Mathematics and Statistics, Memorial University of Newfoundland, St. John's, NL A1C 5S7, Canada}
\email{jxiao@mun.ca}


\date{}



\begin{abstract}
An affine rearrangement inequality is established which strengthens
and implies the recently obtained affine P\'olya--Szeg\"o
symmetrization principle for functions on $\mathbb{R}^n$. Several
applications of this new inequality are derived. In particular, a
sharp affine logarithmic Sobolev inequality is established which is
stronger than its classical Euclidean counterpart.
\end{abstract}
\maketitle

\section{Introduction}

\setcounter{abschnitt}{1} \alpheqn

The classical P\'olya--Szeg\"o principle \cite{PolSzeg51} states
that the $L^p$ norm of the gradient of a function on $\mathbb{R}^n$
does not increase under symmetric rearrangement. It plays a
fundamental role in the solution to a number of variational problems
in different areas such as isoperimetric inequalities, optimal forms
of Sobolev inequalities, and sharp a priori estimates of solutions
to second-order elliptic or parabolic boundary value problems; see,
for example, \cite{brothziem88, kawohl85, kawohl86, kesavan06,
talenti76,talenti93} and the references therein. In recent years,
many important generalizations  and variations have been obtained
(see, e.g., \cite{burchard97, cianchi00, cianfus02, cianfus06a,
cianfus06b, cianfus08, esptromb04, fervolp04}).

Based on the seminal work of Zhang \cite{Zhang99}, a full affine
analogue of the classical P\'olya--Szeg\"o principle was recently
established by  Lutwak, Yang, and Zhang \cite{LYZ2002b} (for $1 \leq
p < n)$ and by Cianchi, Lutwak, Yang, and Zhang \cite{clyz09} (for
general $p \geq 1$). In this remarkable affine rearrangement
inequality, an $L^p$ affine energy replaces the standard $L^p$ norm
of the gradient leading to an inequality which is significantly
stronger than its classical Euclidean counterpart. Moreover, Lutwak,
Yang, and Zhang \cite{LYZ2002b} and Cianchi et al. \cite{clyz09}
obtained new sharp affine Sobolev, Moser--Trudinger and
Morrey--Sobolev inequalities by applying their affine
P\'olya--Szeg\"o principle, thereby demonstrating the power of this
new affine symmetrization inequality.

In this article we establish a new affine P\'olya--Szeg\"o type
inequality which strengthens and directly implies the affine
P\'olya--Szeg\"o principle of Cianchi, Lutwak, Yang, and Zhang. We
will show that an asymmetric $L^p$ affine energy, which takes
asymmetric parts of directional derivatives into account, leads to a
stronger inequality. As an application of our affine rearrangement
inequality we strengthen the previously known affine
Moser--Trudinger and Morrey--Sobolev inequalities of Cianchi et al.
and recover recent results by the first two authors \cite{hs2} on
asymmetric affine Sobolev inequalities. Among further applications
is a new sharp affine logarithmic Sobolev inequality which is
stronger than the classical Euclidean logarithmic Sobolev
inequality.

For $p \in [1,\infty]$ and $n \geq 2$, let $W^{1,p}(\mathbb{R}^n)$
denote the space of real-valued $L^p$ functions on $\mathbb{R}^n$
with weak $L^p$ partial derivatives. We use $|\cdot|$ to denote the
standard Euclidean norm on $\mathbb{R}^n$ and we write $\|f\|_p$ for
the usual $L^p$ norm of a function $f$ on $\mathbb{R}^n$. For $f \in
W^{1,p}(\mathbb{R}^n)$, we set
\[\|\nabla f\|_p := \left ( \int_{\mathbb{R}^n}|\nabla f|^p\,dx \right )^{1/p}.  \]
Given any $f \in W^{1,p}(\mathbb{R}^n)$ such that $\{x \in
\mathbb{R}^n: |f(x)| > t\}$ has finite Lebesgue measure for every $t
> 0$, its {\it distribution function} $\mu_f\!: (0,\infty)
\rightarrow [0,\infty)$ is defined by
\[\mu_f(t)=V(\{x \in \mathbb{R}^n: |f(x)| > t\}),   \]
where $V$ denotes Lebesgue measure on $\mathbb{R}^n$. The {\it
symmetric rearrangement} of $f$ is the function
$f^{\mbox{\begin{tiny}$\bigstar$\end{tiny}}}:\mathbb{R}^n\to[0,\infty]$
defined by
\[f^{\mbox{\begin{tiny}$\bigstar$\end{tiny}}}(x)=\sup \{t > 0: \mu_f(t) > \kappa_n|x|\}.  \]
Here, $\kappa_n=\pi^{n/2}\!/\Gamma(1+\frac{n}{2})$ denotes the
volume of the Euclidean unit ball in \nolinebreak $\mathbb{R}^n$.

The classical P\'olya--Szeg\"o principle states that if $f \in
W^{1,p}(\mathbb{R}^n)$ for some $p \geq 1$, then
$f^{\mbox{\begin{tiny}$\bigstar$\end{tiny}}} \in
W^{1,p}(\mathbb{R}^n)$ and
\begin{equation} \label{polseg}
\|\nabla f^{\mbox{\begin{tiny}$\bigstar$\end{tiny}}}\|_p  \leq
\|\nabla f\|_p.
\end{equation}

In the affine P\'olya--Szeg\"o inequality the $L^p$ norm of the
Euclidean length of the gradient is replaced by an affine invariant
of functions, the (symmetric) $L^p$ {\it affine energy}, defined,
for $f \in W^{1,p}(\mathbb{R}^n)$, by
\[\mathcal{E}_p(f)=c_{n,p}\left(\int_{S^{n-1}}\|\mathrm{D}_u f\|_p^{-n}\,du\right)^{-1/n},\]
where $c_{n,p}=(n\kappa_n)^{1/n} (
\frac{n\kappa_n\kappa_{p-1}}{2\kappa_{n+p-2}} )^{1/p}$ and
$\mathrm{D}_u f$ is the directional derivative of $f$ in the
direction $u$. Note that $c_{n,p}$ is chosen such that if $f \in
W^{1,p}(\mathbb{R}^n)$, then
\begin{equation} \label{epfstar}
\mathcal{E}_p(f^{\mbox{\begin{tiny}$\bigstar$\end{tiny}}})=\|\nabla
f^{\mbox{\begin{tiny}$\bigstar$\end{tiny}}}\|_p.
\end{equation}
We emphasize the remarkable and important fact that
$\mathcal{E}_p(f)$ is invariant under volume preserving affine
transformations on $\mathbb{R}^n$. In contrast, $\|\nabla f\|_p$ is
invariant only under rigid motions.

The affine P\'olya--Szeg\"o principle established by Cianchi et al.
\cite{clyz09} states that if $f \in W^{1,p}(\mathbb{R}^n)$ for some
$p \geq 1$, then
\begin{equation} \label{affpolszeg}
\mathcal{E}_p(f^{\mbox{\begin{tiny}$\bigstar$\end{tiny}}}) \leq
\mathcal{E}_p(f).
\end{equation}
It was shown in \cite{LYZ2002b} that
\begin{equation} \label{comparison}
\mathcal{E}_p(f) \leq \|\nabla f\|_p,
\end{equation}
with equality if and only if $\|\mathrm{D}_uf\|_p$ is independent of
$u \in S^{n-1}$. Thus, by (\ref{epfstar}), the affine inequality
(\ref{affpolszeg}) is significantly stronger than its classical
Euclidean counterpart (\ref{polseg}).

Define the {\it asymmetric $L^p$ affine energy} by
\[\mathcal{E}_p^+(f)=2^{1/p}c_{n,p}\left(\int_{S^{n-1}}\|\mathrm{D}_u^+ f\|_p^{-n}\,du\right)^{-1/n},\]
where $\mathrm{D}^+_u f(x)=\max\{\mathrm{D}_u f(x),0\}$ denotes the
positive part of the directional derivative of $f$ in the direction
$u$. Observe that only the even part of the directional derivatives
of $f$ contribute to $\mathcal{E}_p(f)$, while in
$\mathcal{E}_p^+(f)$ also asymmetric parts are accounted for.

The main result of this article is the following:

\begin{theorem} \label{main} If $p \geq
1$ and $f \in W^{1,p}(\mathbb{R}^n)$, then
$f^{\mbox{\begin{tiny}$\bigstar$\end{tiny}}} \in
W^{1,p}(\mathbb{R}^n)$ and
\begin{equation} \label{hspolszeg}
\mathcal{E}_p^+(f^{\mbox{\begin{tiny}$\bigstar$\end{tiny}}}) \leq
\mathcal{E}_p^+(f).
\end{equation}
\end{theorem}

\vspace{0.3cm}

In \cite{hs2} it was shown that
\begin{equation} \label{comparison2}
\mathcal{E}_p^+(f) \leq \mathcal{E}_p(f),
\end{equation}
with equality if and only if $\|\mathrm{D}_u^+ f\|_p$ is an even
function on $S^{n-1}$. Thus, the new affine P\'olya--Szeg\"o
inequality (\ref{hspolszeg}) is stronger than inequality
(\ref{affpolszeg}) of Cianchi et al. In particular, inequality
(\ref{hspolszeg}) is also stronger than the classical
P\'olya--Szeg\"o inequality (\ref{polseg}).

In the proof of Theorem 1 critical use is made of a new affine
isoperimetric inequality recently established by the first two
authors \cite{hs}. We will apply this crucial geometric inequality
to convex bodies (associated with the given function) which occur as
solutions to a family of (normalized) $L^p$ Minkowski problems.
These techniques clearly demonstrate that there are deep connections
between the affine geometry of convex bodies and sharp affine
functional inequalities (see also \cite{clyz09, LYZ2002b, LYZ2006,
Zhang99}). The background material on the geometric core of Theorem
1 will be discussed in detail in Sections 3 and 4.

The classical P\'olya--Szeg\"o principle has important applications
to a large class of variational problems, for example, it reduces
the proof of sharp Sobolev inequalities to a considerably more
manageable one-dimensional problem. It was shown in \cite{clyz09}
and \cite{LYZ2002b} that the affine P\'olya--Szeg\"o inequality
(\ref{affpolszeg}) provides a similar unified approach to affine
functional inequalities. In particular, sharp affine versions of
$L^p$ Sobolev, Moser--Trudinger and Morrey--Sobolev inequalities
were derived from inequality (\ref{affpolszeg}), all of which are
stronger than their Euclidean counterparts.

In Section 6 we aim to complete the picture given in \cite{clyz09},
\cite{LYZ2002b} and \cite{Zhang99} by deriving new sharp
(asymmetric) affine versions of a number of fundamental functional
inequalities such as $L^p$ Sobolev inequalities, Nash's inequality,
logarithmic Sobolev inequalities and Gagliardo--Nirenberg
inequalities. As an example, we state here our affine version of the
sharp $L^p$ logarithmic Sobolev inequality (the Euclidean analogue
is due to Del Pino and Dolbeault \cite{deldol03}).

\begin{corollary} If $f \in W^{1,p}(\mathbb{R}^n)$, with $1 \leq p
< n$, such that $\|f\|_p=1$, then
\begin{equation} \label{afflogsob}
\int_{\mathbb{R}^n} |f|^p \log |f|\,dx \leq \frac{n}{p} \log \left (
b_{n,p}\, \mathcal{E}_p^+(f)  \right ).
\end{equation}
For $p>1$, the optimal constant $b_{n,p}$ is given by
\begin{equation} \label{optlogsob}
b_{n,p}=\left ( \frac{p}{n} \right )^{1/p}\left (\frac{p-1}{e}
\right )^{1-1/p} \left (
\frac{\Gamma(1+\frac{n}{2})}{\pi^{n/2}\Gamma(1+\frac{n(p-1)}{p})}
\right )^{1/n}
\end{equation}
and $b_{n,1}=\lim_{p\rightarrow 1}b_{n,p}$. If $p=1$, equality holds
in (\ref{afflogsob}) for characteristic functions of ellipsoids and
for $p>1$ equality is attained when
\begin{equation} \label{extrafflogsob}
f(x)=\frac{\pi^{n/2}\Gamma(1+\frac{n}{2})}{a^{n(p-1)/p}\Gamma(1+\frac{n(p-1)}{p})}\exp
\left ( - \frac{1}{a}|\phi(x-x_0)|^{p/(p-1)} \right ),
\end{equation}
with $a>0$, $\phi \in \mathrm{SL}(n)$ and $x_0 \in \mathbb{R}^n$.
\end{corollary}

\section{Background material}

\setcounter{abschnitt}{2} \reseteqn \alpheqn

For quick later reference we recall in this section some background
material from the $L^p$ Brunn--Minkowski theory of convex bodies.
This theory has its origins in the work of Firey from the 1960's and
has expanded rapidly over the last decade (see, e.g.,
\cite{Campi:Gronchi02a, Chou:Wang:06, hs, hlyz, ludwig03,
Ludwig:Minkowski, centro, lutwak86, Lutwak93b, Lutwak96,
Lutwak:Oliker, LYZ2000a, LYZ2000b, lyz04c, LYZ2006}). We will also
list some basic, and for the most part well known, facts from real
analysis needed in the proof of Theorem \ref{main}.

A {\it convex body} is a compact convex subset of $\mathbb{R}^n$
with nonempty interior. We write $\mathcal{K}^n$ for the set of
convex bodies in $\mathbb{R}^n$ endowed with the Hausdorff metric
and we denote by $\mathcal{K}_{\mathrm{o}}^n$ the set of convex
bodies containing the origin in their interiors. Each nonempty
compact convex set $K$ is uniquely determined by its {\it support
function} $h(K,\cdot)$, defined by $h(K,x)=\max\{x\cdot y:\,\,y\in
K\}$, $x\in\mathbb{R}^n$, where $x\cdot y$ denotes the usual inner
product of $x$ and $y$ in $\mathbb{R}^n$. Note that $h(K,\cdot)$ is
positively homogeneous of degree one and subadditive. Conversely,
every function with these properties is the support function of a
unique compact convex set.

If $K \in \mathcal{K}_{\mathrm{o}}^n$, then the polar body $K^*$ of
$K$ is defined by
\[K^*=\{x \in \mathbb{R}^n: x \cdot y \leq 1 \mbox{ for all } y \in K\}.  \]
From the polar formula for volume it follows that the
$n$-dimensional Lebesgue measure $V(K^*)$ of the polar body $K^*$
can be computed by
\begin{equation}\label{polvol}
V(K^*)=\frac{1}{n}\int_{S^{n-1}}h(K,u)^{-n}\,du,\end{equation} where
integration is with respect to spherical Lebesgue measure.

If $M$ and $N$ are compact sets in $\mathbb{R}^n$, then the
Brunn--Minkowski inequality (see, e.g., \cite{Schneider:CB}) states
that
\[V(M + N)^{1/n} \geq V(M)^{1/n} + V(N)^{1/n}.  \]
Here, $M + N = \{x + y: x \in M \mbox{ and } y \in N\}$. For a
compact set $M$ and a convex body $K$ in $\mathbb{R}^n$, define the
mixed volume $V_1(M,K)$ by
\[nV_1(M,K)=\liminf \limits_{\varepsilon \rightarrow 0^+} \frac{V(M + \varepsilon K) - V(M)}{\varepsilon}. \]
The Brunn-Minkowski inequality immediately gives the Minkowski
inequality
\begin{equation}\label{extmink}
V_1(M,K)^n \geq V(M)^{n-1} V(K).
\end{equation}

If the boundary $\partial M$ of $M$ is a $C^1$ submanifold of
$\mathbb{R}^n$, then
\begin{equation}\label{intrep}
V_1(M,K)=\frac{1}{n}\int_{\partial M}
h(K,\nu(x))\,d\mathcal{H}^{n-1}(x),
\end{equation}
where $\nu(x)$ is the outward unit normal vector of $\partial M$ at
$x$ and $\mathcal{H}^{n-1}$ denotes $(n-1)$-dimensional Hausdorff
measure (cf.\ \cite[Lemma 3.2]{Zhang99}).

For real $p \geq 1$ and $\alpha, \beta > 0$, the {\it $L^p$
Minkowski--Firey combination} of $K, L \in
\mathcal{K}^n_{\mathrm{o}}$ is the convex body $\alpha \cdot K +_p
\beta \cdot L$ defined by
\[h(\alpha \cdot K +_p \beta \cdot L,\cdot)^p=\alpha h(K,\cdot)^p+\beta h(L,\cdot)^p.   \]
The {\it $L^p$ mixed volume} $V_p(K,L)$ of $K$,
$L\in\mathcal{K}_{\mathrm{o}}^n$ was defined in \cite{Lutwak93b} by
\begin{equation*}
V_p(K,L)=\frac{p}{n}\lim_{\varepsilon\rightarrow
0^+}\frac{V(K+_p\varepsilon\cdot L)-V(K)}{\varepsilon}.
\end{equation*}
Clearly, the diagonal form of $V_p$ reduces to ordinary volume,
i.e., for $K \in \mathcal{K}^n_{\mathrm{o}}$,
\begin{equation} \label{vol1}
V_p(K,K)=V(K).
\end{equation}
It was also shown in \cite{Lutwak93b} that for all convex bodies $K$
and $L$
\begin{equation} \label{intrepsp}
V_p(K,L)=\frac{1}{n}\int_{S^{n-1}}h(L,u)^ph(K,u)^{1-p}\,dS(K,u),
\end{equation}
where the measure $S(K,\cdot)$ on $S^{n-1}$ is the classical surface
area measure of $K$. Recall that for a Borel set $\omega \subseteq
S^{n-1}$, $S(K,\omega)$ is the $(n-1)$-dimensional Hausdorff measure
of the set of all boundary points of $K$ for which there exists a
normal vector of $K$ belonging to $\omega$.

We turn now to the analytical preparations. Let $f\in
L^p(\mathbb{R}^n)$ for some $p\geq1$ and suppose that $V(\{x \in
\mathbb{R}^n: |f(x)|>t\})< \infty$ for every $t > 0$. The {\it
decreasing rearrangement} $f^*:[0,\infty)\to[0,\infty]$ of $f$ is
defined by
\[f^*(s)=\sup\{t\geq0:\,\,\mu_f(t)>s\},\]
where $\mu_f$ denotes the distribution function of $f$. The
symmetric rearrangement
$f^{\mbox{\begin{tiny}$\bigstar$\end{tiny}}}$ of $f$ can now be
written in the form
\[f^{\mbox{\begin{tiny}$\bigstar$\end{tiny}}}(x)=f^*(\kappa_n|x|^n).\]
Note that $f$ and $f^{\mbox{\begin{tiny}$\bigstar$\end{tiny}}}$ are
equimeasureable, i.e.,
$\mu_f=\mu_{f^{\mbox{\begin{tiny}$\bigstar$\end{tiny}}}}$.
Therefore, we have
\begin{equation} \label{supnorm}
\|f\|_{\infty}=f^*(0)=\|f^{\mbox{\begin{tiny}$\bigstar$\end{tiny}}}\|_{\infty}
\end{equation}
and
\begin{equation} \label{sprt}
V(\textnormal{sprt}\,f)=V(\textnormal{sprt}\,f^{\mbox{\begin{tiny}$\bigstar$\end{tiny}}}),
\end{equation}
where $\textnormal{sprt}\,f$ denotes the support of $f$. Moreover,
the equality
\begin{equation} \label{intstar}
\int_{\mathbb{R}^n}\Phi(|f(x)|)\,dx=\int_{\mathbb{R}^n}\Phi(f^{\mbox{\begin{tiny}$\bigstar$\end{tiny}}}(x))\,dx
=\int_0^{\infty}\Phi(f^*(s))\,ds
\end{equation}
holds for every continuous increasing function
$\Phi:[0,\infty)\to[0,\infty)$.

In the proof of Theorem \ref{main}, we will repeatedly apply
Federer's coarea formula (see, e.g., \cite[p.\ 258]{fed}). We state
here a version which is sufficient for our purposes: If $f:
\mathbb{R}^n \rightarrow \mathbb{R}$ is locally Lipschitz and $g:
\mathbb{R}^n \rightarrow [0,\infty)$ is measurable, then, for any
Borel set $A \subseteq \mathbb{R}$,
\begin{equation} \label{coarea}
\int_{f^{-1}(A)\cap \{|\nabla f| >0\}}g(x)\,dx = \int_A
\int_{f^{-1}\{y\}} \frac{g(x)}{|\nabla f(x)
|}\,d\mathcal{H}^{n-1}(x)\,dy.
\end{equation}

\section{$L^p$ Projection Bodies and the $L^p$ Minkowski Problem}

\setcounter{abschnitt}{3} \reseteqn \alpheqn

In the following we collect the critical ingredients for the proof
of Theorem 1. In the Euclidean setting, the P\'olya--Szeg\"o
principle has the classical isoperimetric inequality at its core. In
the affine setting, the geometric tools are an $L^p$ affine
\linebreak isoperimetric inequality established by the first-named
authors \cite{hs} and the solution to the discrete data case of the
normalized $L^p$ Minkowski problem \cite{hlyz}.

The {\it asymmetric $L^p$ projection body} $\Pi_p^+K$ of $K \in
\mathcal{K}^n_{\mathrm{o}}$, first considered in \cite{Lutwak96}, is
the convex body defined by
\begin{equation} \label{defpipp}
h(\Pi_p^+K,u)^p=\int_{S^{n-1}} (u \cdot v)_+^ph(K,v)^{1-p}\,dS(K,v),
\qquad u \in S^{n-1},
\end{equation}
where $(u \cdot v)_+ = \max\{u \cdot v, 0\}$. The (symmetric) {\it
$L^p$ projection body $\Pi_p K$} of $K \in
\mathcal{K}^n_{\mathrm{o}}$, defined in \cite{LYZ2000b}, is
\[\Pi_pK=\mbox{$\frac{1}{2}$} \cdot \Pi_p^+K+_p \mbox{$\frac{1}{2}$} \cdot  \Pi_p^-K,  \]
where $\Pi_p^-K = \Pi_p^+(-K)$. When $p=1$, asymmetric $L^p$
projection bodies (and symmetric $L^p$ projection bodies) coincide
with the classical projection bodies introduced by Minkowski.

Within the Brunn--Minkowski theory, projection bodies have become a
central notion. They arise naturally in a number of different areas
such as functional analysis, stochastic geometry and geometric
tomography. The fundamental affine isoperimetric inequality which
connects the volume of a convex body with that of its polar
projection body is the {\it Petty projection inequality}
\cite{petty67}. This inequality turned out to be far stronger than
the classical isoperimetric inequality and has led to Zhang's affine
Sobolev inequality \cite{Zhang99}.

In the new $L^p$ Brunn--Minkowski theory, establishing an $L^p$
analog of Petty's projection inequality became a major goal. This
was accomplished for the symmetric $L^p$ projection bodies by
Lutwak, Yang, and Zhang \cite{LYZ2000b} (see also Campi and Gronchi
\cite{Campi:Gronchi02a} for an independent approach): If $K \in
\mathcal{K}_{\mathrm{o}}^n$, then
\begin{equation} \label{lppetty}
V(K)^{n/p-1}V(\Pi_p^*K) \leq
\left(\frac{\kappa_n\kappa_{p-1}}{\kappa_{n+p-2}}\right)^{n/p},
\end{equation}
with equality if and only if $K$ is an ellipsoid centered at the
origin. This \linebreak inequality forms the geometric core of the
affine P\'olya--Szeg\"o principle (\ref{affpolszeg}) of Cianchi et
al. \cite{clyz09}.

Recently the first two authors \cite{hs} established a stronger
$L^p$ Petty projection inequality for asymmetric $L^p$ projection
bodies:

\begin{theorem} \label{hsinequ}
If $p > 1$ and $K \in \mathcal{K}_{\mathrm{o}}^n$, then
\begin{equation} \label{hspetty}
V(K)^{n/p-1}V(\Pi_p^{+,*}K) \leq
\left(\frac{\kappa_n\kappa_{p-1}}{\kappa_{n+p-2}}\right)^{n/p},
\end{equation}
where equality is attained if $K$ is an ellipsoid centered at the
origin.
\end{theorem}

Although this inequality was formulated in \cite{hs} for dimensions
$n\geq3$, we remark that it also holds true in dimension $n = 2$.
The proof is verbally the same as the one given in \cite{hs}.

It was also shown in \cite{hs} that inequality (\ref{hspetty}), for
$p > 1$, strengthens and directly implies inequality (\ref{lppetty})
of Lutwak, Yang, and Zhang: If $K \in \mathcal{K}_{\mathrm{o}}^n$,
then
\begin{equation*}
V(\Pi_p^* K) \leq V(\Pi_p^{+,*}K).
\end{equation*}
If $p$ is not an odd integer, equality holds precisely for
origin-symmetric $K$.

\vspace{0.3cm}

We turn now to the second tool from the geometry of convex bodies
needed in the proof of Theorem 1. The $L^p$ Minkowski problem,
essentially an elliptic Monge--Amp\`ere PDE, deals with the
existence and the uniqueness of convex bodies with prescribed $L^p$
curvature (see, e.g., \cite{Chou:Wang:06, hlyz, lyz04c}). We will
apply our affine isoperimetric inequality (\ref{hspetty}) to the
bodies occurring as solutions to this problem for $p \geq 1$. Since
the geometric inequality assumes that the convex bodies contain the
origin in their interiors, its application is intricate in the
asymmetric situation. Here, the origin can lie on the boundary of
the bodies occurring as the solution to the $L^p$ Minkowski problem.
For this reason we will have to deal with a normalized version of
the discrete-data case of the $L^p$ Minkowski problem (see
\cite[Theorem 1.1]{hlyz}).

\begin{theorem}\label{minkprob} If $\alpha_1,\ldots,\alpha_k>0$ and $u_1,\ldots,u_k\in S^{n-1}$ are
not contained in a closed hemisphere, then, for any $p> 1$, there
exists a polytope $P \in \mathcal{K}_{\mathrm{o}}^n$ such that
\[V(P)h(P,\cdot)^{p-1}\sum_{j=1}^k\alpha_j\delta_{u_j}=S(P,\cdot).\]
\end{theorem}
Here, $\delta_{u_j}$ denotes the probability measure with unit point
mass at $u_j \in  S^{n-1}$.

Two more auxiliary results \cite[Lemmas 2.2 \& 2.3]{lyz04c}
regarding the convex bodies which occur as solutions to the volume
normalized $L^p$ Minkowski problem will also be needed: Let $\mu$ be
a positive Borel measure on $S^{n-1}$, and let $K \in \mathcal{K}^n$
contain the origin. Suppose that
\[V(K)h(K,\cdot)^{p-1}\mu=S(K,\cdot),  \]
and that for some constant $c > 0$,
\[\int_{S^{n-1}}(u\cdot
v)_+^p\,d\mu(v)\geq\frac{n}{c^p} \quad \mbox{for every }u \in
S^{n-1}. \] Then
\begin{equation}\label{volineq1}
V(K)\geq\kappa_n\left(\frac{n}{\mu(S^{n-1})}\right)^{n/p} \qquad
\mbox{and} \qquad K \subset cB_n,
\end{equation}
where $B_n$ denotes the Euclidean unit ball in $\mathbb{R}^n$.

\section{Level sets and asymmetric $L^p$ projection bodies}

\setcounter{abschnitt}{4} \reseteqn \alpheqn

In order to apply the crucial $L^p$ affine isoperimetric inequality
(\ref{hspetty}) in the proof of our main result, it will be
necessary to rewrite $L^p$ gradient integrals over level sets in
terms of $L^p$ mixed volumes. This is done by constructing a family
of convex bodies containing the origin in their interiors by solving
a family of $L^p$ Minkowski problems. In \cite{clyz09}, this was
done by using the normalized {\it even} $L^p$ Minkowski problem. In
the asymmetric situation, we have to deal with the solutions to the
general $L^p$ Minkowski problem. Here, the bodies can contain the
origin in their boundaries. Therefore, we will associate a family of
convex polytopes to a given function which are obtained from the
solution to the discrete-data case of the volume normalized $L^p$
Minkowski problem. This ensures that the polytopes contain the
origin in their interiors and allows us to apply Theorem
\ref{hsinequ}.

We denote by $C_0^{\infty}(\mathbb{R}^n)$ the space of infinitely
differentiable functions on $\mathbb{R}^n$ with compact support. If
$f \in C_0^{\infty}(\mathbb{R}^n)$, then the level set
\[[f]_t:=\{x \in \mathbb{R}^n: |f(x)| \geq t\}  \]
is compact for every $0 < t < \|f\|_{\infty}$. By Sard's theorem,
for almost every $t \in (0,\|f\|_{\infty})$, the boundary
\[\partial [f]_t = \{x \in \mathbb{R}^n: |f(x)| = t\}  \]
of $[f]_t$ is a smooth $(n-1)$-dimensional submanifold of
$\mathbb{R}^n$ with everywhere nonzero normal vector $\nabla f(x)$.

\begin{lemma} \label{l1} Suppose that $p>1$ and $f\in C_0^{\infty}(\mathbb{R}^n)$.
Then, for almost every $t \in (0,\|f\|_{\infty})$, there exists a
sequence of convex polytopes $P_k^t \in \mathcal{K}_{\mathrm{o}}^n$,
$k\in\mathbb{N}$, such that
\[\lim_{k\to\infty}P_k^t=K_f^t\in\mathcal{K}^n\]
and
\begin{equation}\label{volpkt}
\frac{1}{n}\int_{\partial [f]_t}h(K^t_f,\nabla f(x))^p|\nabla
f(x)|^{-1}\,d\mathcal{H}^{n-1}(x)=1.
\end{equation}
Moreover, there exists a convex body $L_f^t \in
\mathcal{K}_{\mathrm{o}}^n$ such that
\[\lim_{k \to \infty} V(P_k^t)^{-1/p}\Pi_p^+P_k^t= L_f^t.\]
\end{lemma}

\begin{proof}  Let $t$ be chosen such that $\partial [f]_t$ is
a smooth manifold with everywhere nonzero normal vector $\nabla
f(x)$ and denote by $\nu(x)=\nabla f(x)/|\nabla f(x)|$ the unit
normal of $\partial[f]_t$ at $x$.

Let $\mu^t$ be the finite positive Borel measure on $S^{n-1}$
satisfying
\begin{equation}\label{defmu}
\int_{S^{n-1}}g(v)\,d\mu^t(v)=\int_{\partial[f]_t}g(\nu(x))|\nabla
f(x)|^{p-1}\,d\mathcal{H}^{n-1}(x)
\end{equation}
for every $g \in C(S^{n-1})$. From
\begin{equation*}\label{surnu}
\{\nu(x):\,\,x\in\partial [f]_t\}=S^{n-1},
\end{equation*}
it follows that for any $u \in S^{n-1}$,
\[\int_{S^{n-1}}(u\cdot v)_+\,d\mu^t(v)=\int_{\partial [f]_t}(u \cdot \nu(x))_+|\nabla f(x)|^{p-1}\,d\mathcal{H}^{n-1}(x)>0.\]
Consequently, the measure $\mu^t$ is not concentrated in a closed
hemisphere.

We can find a sequence $\mu_k^t$, $k\in\mathbb{N}$, of discrete
measures on $S^{n-1}$ whose supports are not contained in a closed
hemisphere and such that $\mu_k^t$ converges weakly to $\mu^t$ as $k
\rightarrow \infty$ (see, e.g., \cite[pp.\ 392-3]{Schneider:CB}). By
Theorem 3, for each $k \in \mathbb{N}$, there exists a polytope
$P_k^t \in \mathcal{K}_{\mathrm{o}}^n$ such that
\begin{equation}\label{defmuk}
V(P_k^t)h(P_k^t,\cdot)^{p-1}\mu_k^t=S(P_k^t,\cdot).
\end{equation}
From definition (\ref{defpipp}), relation (\ref{defmuk}) and the
weak convergence of the measures $\mu_k^t$ it follows that for every
$u \in S^{n-1}$,
\begin{equation} \label{form2}
h\left(V(P_k^t)^{-1/p}\Pi_p^+P_k^t,u\right)^p=\int_{S^{n-1}}(u\cdot
v)_+^p\,d\mu_k^t(v)\longrightarrow\int_{S^{n-1}}(u\cdot
v)_+^p\,d\mu^t(v)>0.
\end{equation}
Since pointwise convergence of support functions implies uniform
convergence (see, e.g., \cite[Theorem 1.8.12]{Schneider:CB}), there
exists a $c>0$ such that for all $k \in \mathbb{N}$,
\begin{equation} \label{prop2}
\int_{S^{n-1}}(u\cdot v)_+^p\,d\mu_k^t(v)>c, \quad\textnormal{for
every }u\in S^{n-1}.
\end{equation}
From (\ref{defmuk}), (\ref{prop2}) and (\ref{volineq1}), it follows
that the sequence $P_k^t$, $k\in\mathbb{N}$, is bounded and that the
volumes $V(P_k^t)$ are also bounded from below. By the Blaschke
selection theorem (see, e.g., \cite[Theorem 1.8.6]{Schneider:CB}),
we can therefore select a subsequence of the $P_k^t$ converging to a
convex body $K^t_f$. After relabeling (if necessary) we may assume
that $\lim_{k\rightarrow\infty}P_k^t=K^t_f$. By definition
(\ref{defmu}), we have
\[\frac{1}{n}\int_{\partial[f]_t}h(K^t_f,\nabla f(x))^p|\nabla
f(x)|^{-1}\,d\mathcal{H}^{n-1}(x)=\frac{1}{n}\int_{S^{n-1}}h(K_f^t,v)^p\,d\mu^t(v).\]
Thus, from the uniform convergence of the support functions
$h(P_k^t,\cdot)$ and the weak convergence of the measures $\mu_k^t$
to the finite measure $\mu^t$, we obtain
\[\lim_{k\to\infty}\frac{1}{n}\int_{S^{n-1}}h(P_k^t,v)^p\,d\mu_k^t(v)=
\frac{1}{n}\int_{\partial[f]_t}h(K^t_f,\nabla f(x))^p|\nabla
f(x)|^{-1}\,d\mathcal{H}^{n-1}(x).\] By (\ref{vol1}),
(\ref{intrepsp}), and relation (\ref{defmuk}), we have for each
$k\in\mathbb{N}$,
\[\frac{1}{n}\int_{S^{n-1}}h(P_k^t,v)^p\,d\mu_k^t(v)=1,\]
which proves (\ref{volpkt}). Finally, we define $h(L_f^t,\cdot)$ by
\begin{equation} \label{defhft}
h(L_f^t,u)^p=\int_{\partial [f]_t}(u\cdot\nabla f(x))_+^p|\nabla
f(x)|^{-1}\,d\mathcal{H}^{n-1}(x),\qquad u\in S^{n-1}.
\end{equation}
From Minkowski's integral inequality, it follows that
$h(L_f^t,\cdot)$ is the support function of a compact convex set.
From definition (\ref{defmu}) and (\ref{form2}), we deduce that
$L_f^t \in \mathcal{K}_{\mathrm{o}}^n$ and that $\lim_{k \to \infty}
V(P_k^t)^{-\frac{1}{p}}\Pi_p^+P_k^t= L_f^t$. \qed
\end{proof}

The following lemma is a special case of Lemma \ref{l1} for
functions with rotational symmetry arising from symmetric
rearrangement.

\begin{lemma} \label{symmlemm} Suppose that $p> 1$ and $f\in C_0^{\infty}(\mathbb{R}^n)$. Then,
for almost every $t\in (0,\|f\|_{\infty})$, there exists a real
number $c_f^t>0$ such that
\begin{equation}\label{defhft2}
h\left(V(c_f^tB_n)^{-1/p}\Pi_p^+(c_f^tB_n),u\right)^p=\int_{\partial[f^{\mbox{\begin{tiny}$\bigstar$\end{tiny}}}]_t}\frac{(u\cdot\nabla
f^{\mbox{\begin{tiny}$\bigstar$\end{tiny}}}(x))_+^p}{|\nabla
f^{\mbox{\begin{tiny}$\bigstar$\end{tiny}}}(x)|}\,d\mathcal{H}^{n-1}(x),
\end{equation}
for every $u\in S^{n-1}$, and
\begin{equation}\label{volpkt2}
\frac{1}{n}\int_{\partial
[f^{\mbox{\begin{tiny}$\bigstar$\end{tiny}}}]_t}h\left(c_f^t
B_n,\nabla
f^{\mbox{\begin{tiny}$\bigstar$\end{tiny}}}(x)\right)^p|\nabla
f^{\mbox{\begin{tiny}$\bigstar$\end{tiny}}}(x)|^{-1}\,d\mathcal{H}^{n-1}(x)=1.
\end{equation}\end{lemma}

\begin{proof} For almost every $t\in (0,\|f\|_{\infty})$, the set $\partial[f^{\mbox{\begin{tiny}$\bigstar$\end{tiny}}}]_t$
is the boundary of a ball of radius $r_t$ with nonvanishing normal
$\nabla f^{\mbox{\begin{tiny}$\bigstar$\end{tiny}}}$. Note that in
this case, $|\nabla f^{\mbox{\begin{tiny}$\bigstar$\end{tiny}}}|$ is
in fact constant on
$\partial[f^{\mbox{\begin{tiny}$\bigstar$\end{tiny}}}]_t$. Define
the real number $c_f^t$ by
\[c_f^t=\left(\kappa_n^{-1}|\nabla f^{\mbox{\begin{tiny}$\bigstar$\end{tiny}}}|^{1-p}r_t^{1-n}\right)^{1/p}.\]
We write $\nu_*(x)=\nabla
f^{\mbox{\begin{tiny}$\bigstar$\end{tiny}}}(x)/|\nabla
f^{\mbox{\begin{tiny}$\bigstar$\end{tiny}}}(x)|$ for the unit normal
vector of $\partial[f^{\mbox{\begin{tiny}$\bigstar$\end{tiny}}}]_t$.
Since for every $g\in C(S^{n-1})$,
\begin{equation*}
\int_{S^{n-1}}g(v)\,d\mathcal{H}^{n-1}(v)=r_t^{1-n}\int_{\partial[
f^{\mbox{\begin{tiny}$\bigstar$\end{tiny}}}]_t}g(\nu_*(x))\,d\mathcal{H}^{n-1}(x),
\end{equation*}
the definition of asymmetric $L^p$ projection bodies (\ref{defpipp})
yields, for $u \in S^{n-1}$,
\begin{equation*}
h\left(V(c_f^tB_n)^{-1/p}\Pi_p^+(c_f^tB_n),u\right)^p
=\frac{r_t^{1-n}}{(c_f^t)^{p}\kappa_n}\int_{\partial[f^{\mbox{\begin{tiny}$\bigstar$\end{tiny}}}]_t}(u\cdot
\nu_*(x))_+^p\, d\mathcal{H}^{n-1}(x).
\end{equation*}
Thus, we obtain (\ref{defhft2}) from the definitions of $c_f^t$ and
$\nu_*(x)$. Finally, we have
\begin{equation*}
\frac{1}{n}\int_{\partial[f^{\mbox{\begin{tiny}$\bigstar$\end{tiny}}}]_t}\!\frac{h(c_f^t
B^n,\nabla
f^{\mbox{\begin{tiny}$\bigstar$\end{tiny}}}(x))^p}{|\nabla
f^{\mbox{\begin{tiny}$\bigstar$\end{tiny}}}(x)|}\,d\mathcal{H}^{n-1}(x)
=\frac{(c_f^t)^p}{n}\int_{\partial[f^{\mbox{\begin{tiny}$\bigstar$\end{tiny}}}]_t}\!|\nabla
f^{\mbox{\begin{tiny}$\bigstar$\end{tiny}}}(x)|^{p-1}\,d\mathcal{H}^{n-1}(x),
\end{equation*}
which yields (\ref{volpkt2}) by the definition of $c_f^t$. \qed
\end{proof}

\section{Proof of the main result}

\setcounter{abschnitt}{5} \reseteqn \alpheqn

We are now in a position to prove our main result. The approach we
use to establish Theorem 1 is based on techniques developed in
\cite{LYZ2002b}.

Before we begin, we want to point out that the asymmetric affine
$L^p$ energy $\mathcal{E}_p^+(f)$ is well defined. This follows from
the fact that $\|\mathrm{D}_u^+ f\|_p$ is positive for each $u \in
S^{n-1}$ and every nontrivial $f \in W^{1,p}(\mathbb{R}^n)$ (see
\cite[Lemma 2]{hs2}).

\vspace{0.3cm}

{\it Proof of Theorem 1} The statement that
$f^{\mbox{\begin{tiny}$\bigstar$\end{tiny}}} \in
W^{1,p}(\mathbb{R}^n)$ whenever $f \in W^{1,p}(\mathbb{R}^n)$ is a
classical fact (see, e.g., \cite{PolSzeg51}). In order to prove
inequality (\ref{hspolszeg}), let us first assume that $p > 1$ and
that $f \in C_0^{\infty}(\mathbb{R}^n)$. Clearly, we may also assume
that $f$ is not identically zero. An application of the coarea
formula (\ref{coarea}) shows that
\begin{equation} \label{duplus}
\|\mathrm{D}^+_u f\|_p^p=\int_0^{\|f\|_{\infty}}\int_{\partial
[f]_t}\frac{(u\cdot\nabla f(x))_+^p}{|\nabla
f(x)|}\,d\mathcal{H}^{n-1}(x)\,dt.
\end{equation}
By Lemma \ref{l1} and (\ref{defhft}), there exists a convex body
$L_f^t \in \mathcal{K}_{\mathrm{o}}^n$ such that
\[\mathcal{E}_p^+(f)^p=2c_{n,p}^p
\left(\int_{S^{n-1}}\left(\int_0^{\|f\|_{\infty}}h(L_f^t,u)^p\,dt\right)^{-n/p}\,du\right)^{-p/n}.\]
Since $h(L_f^t,\cdot)$ is positive, we can apply Minkowski's
inequality for integrals (see, e.g., \cite[p.\ 148]{HLP}), to obtain
\[\mathcal{E}_p^+(f)^p \geq 2c_{n,p}^p
\int_0^{\|f\|_{\infty}}\left(\int_{S^{n-1}}h(L_f^t,u)^{-n}\,du\right)^{-p/n}\,dt.\]
Hence, the volume formula (\ref{polvol}) yields
\begin{equation} \label{est1}
\mathcal{E}_p^+(f)^p \geq
2c_{n,p}^p\int_0^{\|f\|_{\infty}}\left(nV(L_f^{t,*})\right)^{-p/n}\,dt.
\end{equation}
By Lemma \ref{l1}, there exists a sequence of polytopes $P_k^t \in
\mathcal{K}_{\mathrm{o}}^n$ such that
\[ \lim_{k \to \infty} P_k^t=K_f^t \in \mathcal{K}^n \qquad \mbox{and} \qquad \lim_{k \to \infty}
V(P_k^t)^{-1/p}\Pi_p^+P_k^t= L_f^t.\] Thus, an application of
Theorem \ref{hsinequ} shows that
\begin{equation} \label{est2}
(nV(L_f^{t,*}))^{-p/n}=\lim_{k \rightarrow
\infty}(nV(P_k^t)^{n/p}V(\Pi_p^{+,*}P_k^t))^{-p/n} \geq
e_{n,p}V(K_f^t)^{-p/n},
\end{equation}
where
\[e_{n,p}=\frac{\kappa_{n+p-2}}{n^{p/n}\kappa_n\kappa_{p-1}} .\] From (\ref{est1})
and (\ref{est2}), we deduce
\begin{equation} \label{inequkette1}
\mathcal{E}_p^+(f)^p \geq
n\kappa_n^{p/n}\int_0^{\|f\|_{\infty}}V(K_f^t)^{-p/n}\,dt.
\end{equation}
By (\ref{volpkt}) and H\"older's integral inequality, we have
\begin{equation*}
\left(\int_{\partial [f]_t}|\nabla
f(x)|^{-1}\,d\mathcal{H}^{n-1}(x)\right)^{(p-1)/p} \geq n^{1-1/p}\,
V_1([f]_t,K_f^t),
\end{equation*}
where we have used representation (\ref{intrep}) for the mixed
volume $V_1([f]_t,K_f^t)$. From the Minkowski inequality
(\ref{extmink}), we deduce further that
\begin{equation} \label{absch1}
\left(\int_{\partial [f]_t}|\nabla
f(x)|^{-1}\,d\mathcal{H}^{n-1}(x)\right)^{(p-1)/p} \geq n^{1-1/p}\,
\mu_f(t)^{(n-1)/n}V(K_f^t)^{1/n}.
\end{equation}
By the coarea formula (\ref{coarea}), we have for almost every $t$,
\begin{equation}\label{coabl}
\mu_f(t)=V([f]_t \cap \{\nabla
f=o\})+\int_t^{\|f\|_{\infty}}\int_{\partial [f]_s}|\nabla
f(x)|^{-1}\,d\mathcal{H}^{n-1}(x)\,ds.
\end{equation}
Since $\mu_f$ is the sum of two nonincreasing functions, we obtain
\begin{equation} \label{absch17}
-\mu_f(t)'\geq\int_{\partial[f]_t}|\nabla
f(x)|^{-1}\,d\mathcal{H}^{n-1}(x)
\end{equation}
for almost every $t$.

Combining (\ref{absch1}) and (\ref{absch17}), yields the estimate
\begin{equation}\label{absch2}
V(K_f^t)^{-p/n}\geq
n^{p-1}\frac{\mu_f(t)^{p(n-1)/n}}{(-\mu_f'(t))^{p-1}}.
\end{equation}
Thus, by (\ref{inequkette1}) and (\ref{absch2}), we obtain
\[\mathcal{E}_p^+(f)^p\geq n^p\kappa_n^{p/n}\int_0^{\|f\|_{\infty}}\frac{\mu_f(t)^{p(n-1)/n}}{(-\mu_f'(t))^{p-1}}\,dt.\]
It remains to show that
\begin{equation} \label{starequal}
\mathcal{E}_p^+(f^{\mbox{\begin{tiny}$\bigstar$\end{tiny}}})^p=
 n^p\kappa_n^{p/n}\int_0^{\|f\|_{\infty}}\frac{\mu_f(t)^{p(n-1)/n}}{(-\mu_f'(t))^{p-1}}\,dt.
\end{equation}
By (\ref{supnorm}), (\ref{defhft2}), and (\ref{duplus}), we have
\begin{eqnarray*}
\mathcal{E}_p^+(f^{\mbox{\begin{tiny}$\bigstar$\end{tiny}}})^p=2c_{n,p}^p
\left(\int_{S^{n-1}}\left(\int_0^{\|f\|_{\infty}}V(B_f^t)^{-1}h(\Pi_p^+B_f^t,u)^p
dt\right)^{-n/p} du\right)^{-p/n},
\end{eqnarray*}
where $B_f^t$ denotes the ball $c_f^tB_n$ whose existence is
guaranteed by Lemma \ref{symmlemm}. Since $\Pi_p^+B_f^t$ is a ball,
$h(\Pi_p^+B_f^t,\cdot)$ is a constant function on the sphere. Thus,
we obtain as in the first part of the proof,
\begin{equation} \label{inequkette3}
\mathcal{E}_p^+(f^{\mbox{\begin{tiny}$\bigstar$\end{tiny}}})^p =
n\kappa_n^{p/n}\int_0^{\|f\|_{\infty}}V(B_f^t)^{-p/n}\,dt.
\end{equation}
From (\ref{volpkt2}), Minkowski's inequality (\ref{extmink}), and
the fact that $[f^{\mbox{\begin{tiny}$\bigstar$\end{tiny}}}]_t$ and
$B_f^t$ are dilates, we have for almost every $t$ on one hand
\begin{equation*}
\left(\int_{\partial[f^{\mbox{\begin{tiny}$\bigstar$\end{tiny}}}]_t}|\nabla
f^{\mbox{\begin{tiny}$\bigstar$\end{tiny}}}(x)|^{-1}\,d\mathcal{H}^{n-1}(x)\right)^{(p-1)/p}
=
n^{1-1/p}\mu_{f^{\mbox{\begin{tiny}$\bigstar$\end{tiny}}}}(t)^{(n-1)/n}V(B_f^t)^{1/n}
\end{equation*}
and, by (\ref{coabl}) and \cite[Lemma 2.4 \& 2.6]{cianfus02}, on the
other hand
\[-\mu_{f^{\mbox{\begin{tiny}$\bigstar$\end{tiny}}}}(t)'=\int_{\partial[f^{\mbox{\begin{tiny}$\bigstar$\end{tiny}}}]_t}|\nabla
f^{\mbox{\begin{tiny}$\bigstar$\end{tiny}}}(x)|^{-1}\,d\mathcal{H}^{n-1}(x).\]
Hence, the equimeasurability of $f$ and
$f^{\mbox{\begin{tiny}$\bigstar$\end{tiny}}}$ yields
\[V(B_f^t)^{-p/n}= n^{p-1}\frac{\mu_f(t)^{p(n-1)/n}}{(-\mu_f'(t))^{p-1}}.\]
Combining this with (\ref{inequkette3}) proves (\ref{starequal}).
Consequently, we have
\begin{equation} \label{inequglatt}
\mathcal{E}_p^+(f^{\mbox{\begin{tiny}$\bigstar$\end{tiny}}}) \leq
\mathcal{E}_p^+(f)
\end{equation}
for every $p > 1$ and every $f \in C_0^{\infty}(\mathbb{R}^n)$.
Clearly, the case $p = 1$ of inequality (\ref{inequglatt}) can be
obtained by using a limiting argument as $p \rightarrow 1$.

In order to establish inequality (\ref{inequglatt}) for an arbitrary
$f \in W^{1,p}(\mathbb{R}^n)$ whose \linebreak support has positive
measure, consider a sequence $f_k \in C_0^{\infty}(\mathbb{R}^n)$,
$k\in\mathbb{N}$, converging to $f$ in $W^{1,p}(\mathbb{R}^n)$.
Then, for every $k \in \mathbb{N}$,
\begin{equation} \label{inequ117}
\mathcal{E}_p^+(f_k^{\mbox{\begin{tiny}$\bigstar$\end{tiny}}}) \leq
\mathcal{E}_p^+(f_k).
\end{equation}
By Minkowski's integral inequality, $h_f(u):= \|\mathrm{D}_u^+
f\|_p$, $u \in S^{n-1}$, is a support function of a convex body for
every $f \in W^{1,p}(\mathbb{R}^n)$. Thus, the pointwise convergence
$\|\mathrm{D}_u^+ f_k\|_p\to\|\mathrm{D}_u^+ f\|_p$ on $S^{n-1}$,
implies in fact that $\|\mathrm{D}_u^+ f_k\|_p$ converges to
$\|\mathrm{D}_u^+ f\|_p$ uniformly (see, e.g., \cite[Theorem
1.8.12]{Schneider:CB}). Moreover, since $\|\mathrm{D}_u^+ f\|_p$ is
strictly positive on $S^{n-1}$ (see \cite[Lemma 2]{hs}), also
$\|\mathrm{D}_u^+ f_k\|_p^{-n}\to\|\mathrm{D}_u^+ f\|_p^{-n}$
uniformly on $S^{n-1}$. Hence,
\begin{equation} \label{inequ118}
\lim_{k\to\infty}\mathcal{E}_p^+(f_k)=\mathcal{E}_p^+(f).
\end{equation}
On the other hand, the nonexpansivity of symmetric rearrangements
(see, e.g., \cite{chiti79}) implies
$f_k^{\mbox{\begin{tiny}$\bigstar$\end{tiny}}}\to
f^{\mbox{\begin{tiny}$\bigstar$\end{tiny}}}$ in $L^p(\mathbb{R}^n)$.
Thus, the sequence $f_k^{\mbox{\begin{tiny}$\bigstar$\end{tiny}}}$,
$k \in \mathbb{N}$, converges weakly to
$f^{\mbox{\begin{tiny}$\bigstar$\end{tiny}}}$ in
$W^{1,p}(\mathbb{R}^n)$. Since
\[\mathcal{E}_p^+(f_k^{\mbox{\begin{tiny}$\bigstar$\end{tiny}}})=\|\nabla f_k^{\mbox{\begin{tiny}$\bigstar$\end{tiny}}}\|_p \qquad
\mbox{and} \qquad
\mathcal{E}_p^+(f^{\mbox{\begin{tiny}$\bigstar$\end{tiny}}})=\|\nabla
f^{\mbox{\begin{tiny}$\bigstar$\end{tiny}}}\|_p\] and since the
$L^p$ norm of the gradient is lower semicontinuous with respect to
weak convergence in $W^{1,p}(\mathbb{R}^n)$, we obtain
\[\liminf_{k \to \infty}
\mathcal{E}_p^+(f_k^{\mbox{\begin{tiny}$\bigstar$\end{tiny}}})\geq
\mathcal{E}_p^+(f^{\mbox{\begin{tiny}$\bigstar$\end{tiny}}})
\] which, by (\ref{inequ117}) and (\ref{inequ118}), concludes the proof. \qed

\section{Applications of the asymmetric affine P\'olya-Szeg\"o principle}

\setcounter{abschnitt}{6} \reseteqn \alpheqn

In this section we will illustrate how Theorem \ref{main} provides a
direct unified approach to a number of affine functional
inequalities. We will derive sharp affine versions of certain
Gagliardo--Nirenberg inequalities, all of which are stronger than
their Euclidean counterparts.

\vspace{0.4cm}

{\it Affine $L^p$ logarithmic Sobolev inequalities}

\vspace{0.2cm}

The classical sharp logarithmic Sobolev inequality states that if $f
\in W^{1,2}(\mathbb{R}^n)$ such that $\|f\|_2=1$, then
\begin{equation} \label{eulogsob}
\int_{\mathbb{R}^n} |f|^2\log |f|\,dx \leq \frac{n}{2} \log \left (
\left ( \frac{2}{ne\pi}  \right )^{1/2} \|\nabla f\|_2 \right ).
\end{equation}
In this form, the logarithmic Sobolev inequality first appeared in
\cite{weissler78}. However, it is well known that inequality
(\ref{eulogsob}) is equivalent to the logarithmic Sobolev inequality
with respect to Gauss measure due to Stam \cite{stam59} and Gross
\cite{gross75}. Different proofs and extensions of these
inequalities have been the focus of a number of articles (see, e.g.,
\cite{adams79, beckner98, beckner99, carlen91}, and the references
therein).

The natural problem to find a sharp $L^p$ analogue of inequality
(\ref{eulogsob}) was solved by Ledoux \cite{ledoux96} for $p=1$ and
recently by Del Pino and Dolbeault \cite{deldol03} for $1 < p < n$:
If $f \in W^{1,p}(\mathbb{R}^n)$, with $1 \leq p < n$, such that
$\|f\|_p = 1$, then
\begin{equation} \label{lplogsob}
\int_{\mathbb{R}^n} |f|^p \log |f|\,dx \leq \frac{n}{p} \log \left (
b_{n,p}\, \|\nabla f\|_p  \right ),
\end{equation}
where the optimal constant $b_{n,p}$ is given by (\ref{optlogsob}).
Beckner \cite{beckner99} proved that, for $p=1$, the only extremals
in inequality (\ref{lplogsob}) are the characteristic functions of
balls. Carlen \cite{carlen91}, for $p=2$, and Del Pino and Dolbeault
\cite{deldol03}, for general $ 1 < p < n$, showed that equality
holds in (\ref{lplogsob}) if and only if for some $a > 0$ and $x_0
\in \mathbb{R}^n$,
\begin{equation} \label{equlogsob}
f(x)=\frac{\pi^{n/2}\Gamma(1+\frac{n}{2})}{a^{n(p-1)/p}\Gamma(1+\frac{n(p-1)}{p})}\exp
\left ( - \frac{1}{a}|x-x_0|^{p/(p-1)} \right ).
\end{equation}
The first application of our new affine P\'olya--Szeg\"o principle
is a, in light of (\ref{comparison}) and (\ref{comparison2}),
stronger asymmetric affine version of (\ref{lplogsob}), stated in
the Introduction as Corollary 2.

\vspace{0.2cm}

{\it Proof of Corollary 2} By (\ref{intstar}) and (\ref{lplogsob}),
we have
\begin{equation} \label{logsobpr}
\int_{\mathbb{R}^n} |f|^p \log |f|\,dx= \int_{\mathbb{R}^n}
|f^{\mbox{\begin{tiny}$\bigstar$\end{tiny}}}|^p \log
|f^{\mbox{\begin{tiny}$\bigstar$\end{tiny}}}|\,dx
 \leq \frac{n}{p} \log
\left ( b_{n,p}\, \|\nabla
f^{\mbox{\begin{tiny}$\bigstar$\end{tiny}}}\|_p  \right )
\end{equation}
for every $f \in W^{1,p}(\mathbb{R}^n)$ such that $\|f\|_p=1$. Since
\begin{equation*}
\mathcal{E}_p^+(f^{\mbox{\begin{tiny}$\bigstar$\end{tiny}}})
=\left(\int_{0}^{\infty}\left(n\kappa_n^{1/n}s^{(n-1)/n}(-f^{*'}(s))\right)^pds\right)^{1/p}=\|\nabla
f^{\mbox{\begin{tiny}$\bigstar$\end{tiny}}}\|_p,
\end{equation*}
we deduce from (\ref{logsobpr}) and Theorem \ref{main} that
\begin{equation} \label{logsobpr2}
\int_{\mathbb{R}^n} |f|^p \log |f|\,dx \leq \frac{n}{p} \log \left (
b_{n,p}\,
\mathcal{E}_p^+(f^{\mbox{\begin{tiny}$\bigstar$\end{tiny}}}) \right
) \leq  \frac{n}{p} \log \left ( b_{n,p}\, \mathcal{E}_p^+(f)
\right)
\end{equation}
which proves inequality (\ref{afflogsob}). Equality holds in
(\ref{logsobpr}) for any function having the form (\ref{equlogsob})
with $x_0=o$. Any such function is spherically symmetric, so that
equality holds in Theorem 1 and, thus, also in inequality
(\ref{logsobpr2}). Equality continues to hold in (\ref{afflogsob})
for any function of the form (\ref{extrafflogsob}), owing to the
invariance of (\ref{afflogsob}) under volume preserving affine
transformations. $\hfill \qed$

\vspace{0.4cm}

{\it Affine $L^p$ Sobolev inequalities}

\vspace{0.2cm}

The classical sharp $L^p$ Sobolev inequality states that if $f \in
W^{1,p}(\mathbb{R}^n)$, with $1 \leq p < n$, then
\begin{equation} \label{sob}
\|f\|_{p^*} \leq a_{n,p}\, \|\nabla f\|_p,
\end{equation}
where $p^*=np/(n-p)$. The optimal constants $a_{n,p}$ are given by
\[a_{n,p}=n^{-1/p}\left ( \frac{p-1}{n-p}\right )^{1-1/p}
\left (
\frac{\Gamma(n)}{\kappa_n\Gamma(\scriptstyle{\frac{n}{p}})\scriptstyle{\Gamma(n+1-\scriptstyle{\frac{n}{p}})}}\right
)^{1/n} ,\] and go back to Federer and Fleming \cite{fedflem} and
Maz'ya \cite{mazya} for $p=1$ and to Aubin \cite{aubin76} and
Talenti \cite{talenti76} for $p
> 1$. The extremal functions for inequality (\ref{sob}) are the
characteristic functions of balls for $p=1$ and for $p > 1$ equality
is attained if for some $a, b > 0$ and $x_0 \in \mathbb{R}^n$,
\[f(x)=(a+b|(x-x_0)|^{p/(p-1)})^{1-n/p}.  \]
The sharp $L^p$ Sobolev inequality plays a central role in the
theory of partial \linebreak differential equations and functional
analysis. Generalizations of (\ref{sob}) and related problems have
been much studied (see, e.g., \cite{bakledoux95, brothziem88,
cianfus02, cornazvil04, deldol02, kesavan06, LYZ2006, talenti93,
Xiao07}), and the references therein).

A, in light of (\ref{comparison}), stronger affine version of
inequality (\ref{sob}), was established by Zhang \cite{Zhang99} for
$p=1$ and Lutwak, Yang, and Zhang \cite{LYZ2002b} for $1 < p < n$.
It states that if $f \in W^{1,p}$, with $1 \leq p < n$, then
\begin{equation} \label{affsob}
 \|f\|_{p^*} \leq a_{n,p}\, \mathcal{E}_p(f).
\end{equation}
If $p=1$, equality holds in (\ref{affsob}) for characteristic
functions of ellipsoids and for $p>1$ equality is attained when
\begin{equation} \label{extrasob}
f(x)=(a+|\phi(x-x_0)|^{p/(p-1)})^{1-n/p},
\end{equation}
with $a>0$, $\phi\in\textnormal{GL}(n)$ and $x_0\in\mathbb{R}^n$.

A, by (\ref{comparison2}), strengthened asymmetric version of the
affine Sobolev inequality (\ref{affsob}) was recently established by
the first two authors \cite{hs2}. It is now an immediate consequence
of Theorem 1 and (\ref{sob}):

\begin{corollary} If $f \in W^{1,p}(\mathbb{R}^n)$,
with $1 \leq p < n$, then
\begin{equation} \label{aaffsob}
 \|f\|_{p^*} \leq a_{n,p}\, \mathcal{E}_p^+(f).
\end{equation}
If $p=1$, equality holds in (\ref{aaffsob}) for characteristic
functions of ellipsoids and for $p>1$ equality is attained for
functions of the form (\ref{extrasob}).
\end{corollary}

We turn now to the limiting case $p=n$ of inequality (\ref{sob}). It
is well known that functions $f \in W^{1,n}(\mathbb{R}^n)$, whose
support has finite Lebesgue measure, are exponentially summable
(cf., e.g., \cite{trud67}). The sharp Moser--Trudinger inequality
\linebreak \cite{moser70, trud67} states that there exists a
constant $m_n > 0$ such that
\begin{equation} \label{mostrud}
\frac{1}{V(\mathrm{sprt}\,f)}\int_{\mathrm{sprt}\,f}\exp\left(\frac{n\kappa_n^{1/n}|f(x)|}{\|\nabla
f\|_n}\right)^{n/(n-1)}\,dx\leq m_n
\end{equation}
for every $f \in W^{1,n}(\mathbb{R}^n)$ with $0 <
V(\mathrm{sprt}\,f)<\infty$. Inequality (\ref{mostrud}) and its
variants have been the focus of investigations by specialists in
different areas (see, e.g., \cite{beckner93, changyang88, cianchi05,
cianchi08, cohnlu01, flucher92, lin96, ruf05}).

The constant $n \kappa_n^{1/n}$ is optimal, in that inequality
(\ref{mostrud}) would fail for any real number $m_n$ if
$n\kappa_n^{1/n}$ were to be replaced by a larger number. The best
constant $m_n$ is characterized as follows
\[m_n=\sup_{g} \int_0^{\infty} \exp \left ( g(t)^{n/(n-1)}-t \right )\,dt,  \]
where the supremum ranges over all nondecreasing and locally
absolutely continuous functions $g$ on $[0,\infty)$ such that
$g(0)=0$ and $\int_0^{\infty}g'(t)^n\,dt \leq 1$. In
\cite{carlchang86} Carleson and Chang showed that spherically
symmetric extremals do exist for the Moser--Trudinger inequality
(\ref{mostrud}).

An affine version of the Moser--Trudinger inequality, stronger than
(\ref{mostrud}), was recently established by Cianchi et al.
\cite{clyz09}. It states that if $f\in W^{1,n}(\mathbb{R}^n)$ with
$0<V(\mathrm{sprt}\,f)<\infty$, then
\begin{equation} \label{affmostrud}
\frac{1}{V(\mathrm{sprt}\,f)}\int_{\mathrm{sprt}\,f}\exp\left(\frac{n\kappa_n^{1/n}|f(x)|}{\mathcal{E}_n(f)}\right)^{n/(n-1)}\,dx\leq
m_n.
\end{equation}
The constants $n\kappa_n^{1/n}$ and $m_n$ are again best possible.
Composing any extremal $f$ for the Moser--Trudinger inequality with
any element of $\mathrm{GL}(n)$ will also yield an extremal for
inequality (\ref{affmostrud}).

From Theorem \ref{main} we can derive a strengthened asymmetric
version of the affine Moser--Trudinger inequality
(\ref{affmostrud}):

\begin{corollary} \label{cormos} If $f\in W^{1,n}(\mathbb{R}^n)$ with
$0<V(\mathrm{sprt}\,f)<\infty$, then
\begin{equation}\label{aaffmos}
\frac{1}{V(\mathrm{sprt}\,f)}\int_{\mathrm{sprt}\,f}\exp\left(\frac{n\kappa_n^{1/n}|f(x)|}{\mathcal{E}_n^+(f)}\right)^{n/(n-1)}\,dx\leq
m_n.
\end{equation}
The constant $n\kappa_n^{1/n}$ is optimal, in that (\ref{aaffmos})
would fail for any real number $m_n$ if $n\kappa_n^{1/n}$ were to be
replaced by a larger number. Composing any extremal $f$ for
inequality (\ref{mostrud}) with any element of $\mathrm{GL}(n)$ will
also yield an extremal for inequality (\ref{aaffmos}).
\end{corollary}

We will omit the proof of Corollary \ref{cormos} since it is almost
verbally the same as the one for inequality (\ref{affmostrud}) given
in \cite{clyz09} when \cite[Theorem 2.1]{clyz09} is replaced by
Theorem 1.

\vspace{0.2cm}

Finally, we come to the case $p > n$. The sharp Morrey--Sobolev
inequality \cite{talenti94} states that if $f \in
W^{1,p}(\mathbb{R}^n)$, $p > n$, such that $V(\mathrm{sprt}\,f) <
\infty$, then
\begin{equation} \label{morsob}
\|f\|_{\infty} \leq
\alpha_{n,p}\,V(\mathrm{sprt}\,f)^{(p-n)/np}\|\nabla f\|_p,
\end{equation}
where the optimal constant $\alpha_{n,p}$ is given by
\[\alpha_{n,p} = n^{-1/p}\kappa_n^{-1/n} \left ( \frac{p-1}{p-n} \right )^{(p-1)/p}.  \]
Equality holds in inequality (\ref{morsob}) if for some $a, b
\in\mathbb{R}$ and $x_0\in\mathbb{R}^n$,
\[f(x)=a\left(1-|b(x-x_0)|^{(p-n)/(p-1)}\right)_+.  \]
The affine counterpart of (\ref{morsob}) established by Cianchi et
al. \cite{clyz09} states that
\begin{equation} \label{affmorsob}
\|f\|_{\infty} \leq
\alpha_{n,p}\,V(\mathrm{sprt}\,f)^{(p-n)/np}\mathcal{E}_p(f).
\end{equation}
By (\ref{comparison}), the affine inequality (\ref{affmorsob}) is
significantly stronger than (\ref{morsob}). As an immediate
consequence of Theorem 1, (\ref{supnorm}), (\ref{sprt}) and
(\ref{morsob}) we obtain the following strengthened asymmetric
affine Morrey--Sobolev inequality:

\begin{corollary} If $f\in W^{1,p}(\mathbb{R}^n)$, $p > n$, such that
$V(\mathrm{sprt}\,f)<\infty$, then
\begin{equation}\label{aaffmorr}
\|f\|_{\infty} \leq
\alpha_{n,p}\,V(\mathrm{sprt}\,f)^{(p-n)/np}\mathcal{E}_p^+(f).
\end{equation}
Equality is attained in (\ref{aaffmorr}) if for some
$a\in\mathbb{R}$, $x_0\in\mathbb{R}^n$, and
$\phi\in\textnormal{GL}(n)$,
\[f(x)=a\left(1-|\phi(x-x_0)|^{(p-n)/(p-1)}\right)_+.\]
\end{corollary}

For $f\in W^{1,\infty}(\mathbb{R}^n)$ define the {\it asymmetric
$L^{\infty}$ affine energy} by
\[\mathcal{E}_\infty^+(f)=\left(\int_{S^{n-1}}\|\mathrm{D}_u^+ f\|_{\infty}^{-n}\,du\right)^{-1/n}.\]
We are now in a position to prove the following Faber-Krahn type
inequality.
\begin{corollary}If $f\in W^{1,\infty}(\mathbb{R}^n)$ such that
$V(\mathrm{sprt}\,f)<\infty$, then
\begin{equation}\label{aaffmorrinfty}
\|f\|_{\infty} \leq
\kappa_n^{-1/n}V(\mathrm{sprt}\,f)^{1/n}\mathcal{E}_{\infty}^+(f).
\end{equation}
Equality is attained in (\ref{aaffmorrinfty}) if for some
$a\in\mathbb{R}$, $x_0\in\mathbb{R}^n$, and
$\phi\in\textnormal{GL}(n)$,
\[f(x)=a\left(1-|\phi(x-x_0)|\right)_+.\]
\end{corollary}

\begin{proof}Needless to say, we may take a limit in
(\ref{aaffmorr}) to get the desired estimate. Indeed, by Fatou's
lemma we get
\begin{eqnarray*}
\|f\|_\infty&\le&\kappa_n^{-1/n}V(\mathrm{sprt}\,f)^{1/n}\limsup_{q\to\infty}\mathcal{E}_q^+(f^\star)\\
&\le& \kappa_n^{-1/n}V(\mathrm{sprt}\,f)^{1/n}\Big(\liminf_{q\to\infty}\int_{S^{n-1}}\|D_v^+f\|_q^{-n}\,dv\Big)^{-\frac1n}\\
&\le&\kappa_n^{-1/n}V(\mathrm{sprt}\,f)^{1/n}\Big(\int_{S^{n-1}}\liminf_{q\to\infty}\|D_v^+f\|_q^{-n}\,dv\Big)^{-\frac1n}\\
&\le&
\kappa_n^{-1/n}V(\mathrm{sprt}\,f)^{1/n}\Big(\int_{S^{n-1}}\|D_v^+f\|_\infty^{-n}\,dv\Big)^{-\frac1n}.
\end{eqnarray*}
The corresponding equality case can be verified by a straightforward
computation.
\end{proof}

\vspace{0.4cm}

{\it Affine Nash inequality}

\vspace{0.2cm}

Nash's inequality in its optimal form, established by Carlen and
Loss \cite{carlenloss93}, states that if $f \in L^1(\mathbb{R}^n)
\cap W^{1,2}(\mathbb{R}^n)$, then
\begin{equation} \label{nash}
\|f\|_2^{1+2/n}\leq \beta_n\,\|\nabla f\|_2\,\|f\|_1^{2/n}.
\end{equation}
The best constant $\beta_n$ is given by
\[\beta_n^2=\frac{2\left (1+\frac{n}{2}\right )^{1+n/2}}{n\lambda_n \kappa^{2/n}},  \]
where $\lambda_n$ denotes the first nonzero Neumann eigenvalue of
the Laplacian $-\Delta$ on radial functions on $B_n$. There is
equality in (\ref{nash}) if and only if up to normalization and
scaling
\[f(x)=\left \{ \begin{array}{ll} u(|x-x_0|)-u(1), \,\,\, & \mbox{if } |x| \leq 1 \\ 0, & \mbox{if } |x| \geq 1,   \end{array} \right .  \]
for some $x_0 \in \mathbb{R}^n$. Here, $u$ is the normalized
eigenfunction of the Neumann Laplacian on $B_n$ with eigenvalue
$\lambda_n$. Note the striking feature that all of the extremals
have compact support. Nash's inequality and its variants have proven
to be very useful in a number of contexts (see, e.g.,
\cite{bakledoux95, beckner98, bendmah07, humbert05} and the
references therein).

From an application of Theorem \ref{main} together with
(\ref{nash}), we immediately obtain a new stronger asymmetric affine
version of Nash's inequality.

\begin{corollary} If $f \in L^1(\mathbb{R}^n) \cap W^{1,2}(\mathbb{R}^n)$, then
\begin{equation} \label{anash}
\|f\|_2^{1+2/n}\leq \beta_n\,\mathcal{E}_2^+(f)\|f\|_1^{2/n}.
\end{equation}
Equality is attained in (\ref{anash}) if up to normalization and
scaling
\[f(x)=\left \{ \begin{array}{ll} u(|\phi(x-x_0)|)-u(1), \,\,\, & \mbox{if } |x| \leq 1 \\ 0, & \mbox{if } |x| \geq 1,   \end{array} \right .  \]
for some $\phi \in \mathrm{SL}(n)$ and $x_0 \in \mathbb{R}^n$.
\end{corollary}

\vspace{0.4cm}

{\it Affine Gagliardo--Nirenberg inequalities}

\vspace{0.2cm}

The $L^p$ Sobolev inequality (\ref{sob}), Nash's inequality
(\ref{nash}) and the logarithmic Sobolev inequality (\ref{lplogsob})
are special cases (a limiting case, respectively) of the
Gagliardo--Nirenberg inequalities
\begin{equation} \label{gagnir}
\|f\|_r \leq C_n(p,r,s)\, \|\nabla f
\|_p^{\theta}\,\|f\|_s^{1-\theta},
\end{equation}
where $1 < p < n$, $s < r \leq p^*$, and $\theta \in (0,1)$ is
determined by scaling invariance. While inequality (\ref{gagnir})
can be deduced from (\ref{sob}) with the help of H\"older's
inequality, the computation of the optimal constants $C_n(p,r,s)$ is
an open problem in general. A breakthrough was recently achieved by
Del Pino and Dolbeault \cite{deldol02, deldol03} (see also
\cite{cornazvil04} for a different approach). They obtained the
following sharp one-parameter family of inequalities: Suppose that
$1< p < n$, $p < q \leq p(n-1)/(n-p)$ and let
\begin{equation} \label{const17}
r=\frac{p(q-1)}{p-1} \qquad \mbox{and} \qquad
\theta=\frac{n(q-p)}{(q-1)(np-(n-p)q)}.
\end{equation}
Then, for every $f \in D^{p,q}(\mathbb{R}^n)$,
\begin{equation} \label{deldolgag}
\|f\|_r \leq \gamma_{n,p,q}\, \|\nabla f\|_p^{\theta}\,
\|f\|_q^{1-\theta},
\end{equation}
where $D^{p,q}$ denotes the completion of the space of smooth
compactly supported functions with respect to the norm
$\|\cdot\|_{p,q}$ defined by $\|f\|_{p,q}=\|\nabla f\|_p + \|f\|_q$.
The optimal constant $\gamma_{p,q}$ is given by
\[\gamma_{n,p,q}=\left (\frac{q-p}{p\sqrt{\pi}} \right )^{\theta}\! \left (  \frac{pq}{n(q-p)} \right
)^{\theta/p}\! \left ( \frac{\delta}{pq} \right )^{1/r}\! \left (
\frac{\Gamma\left (\frac{q(p-1)}{q-p} \right ) \Gamma \left (1 +
\frac{n}{2} \right ) }{\Gamma \left ( \frac{\delta(p-1)}{p(q-p)}
\right )\Gamma \left ( 1 + \frac{n(p-1)}{p}  \right )} \right
)^{\theta/n},
\] where $\delta=np-q(n-p)$. Equality holds in (\ref{deldolgag}) if and only if
for some $a \in \mathbb{R}$, $b > 0$ and $x_0 \in \mathbb{R}^n$,
\[f(x)=a \left (1+b|x-x_0|^{p/(p-1)} \right )^{-(p-1)/(q-p)}.  \]
Observe that for $q=p(n-1)/(n-p)$, we have $\theta=1$ and inequality
(\ref{deldolgag}) becomes the sharp $L_p$ Sobolev inequality
(\ref{sob}) of Aubin and Talenti. On the other hand, the logarithmic
Sobolev inequality (\ref{lplogsob}) corresponds to the limit $q
\rightarrow p$ in (\ref{deldolgag}). Thus the Gagliardo--Nirenberg
inequalities (\ref{deldolgag}) interpolate between the sharp $L^p$
Sobolev and the logarithmic Sobolev inequalities.

We conclude this final section with a strengthened family of
asymmetric affine Gagliardo--Nirenberg inequalities which
interpolate between inequalities (\ref{aaffsob}) and
(\ref{afflogsob}). It is an immediate corollary of Theorem
\ref{main} and (\ref{deldolgag}):

\begin{corollary} Let $1< p < n$, $p < q \leq p(n-1)/(n-p)$ and let $r, \theta$ be given by (\ref{const17}).
If $f \in D^{p,q}(\mathbb{R}^n)$, then
\begin{equation} \label{agagnir}
\|f\|_r \leq \gamma_{n,p,q}\, \mathcal{E}_p^+(f)^{\theta}\,
\|f\|_q^{1-\theta}.
\end{equation}
Equality is attained in (\ref{agagnir}) if for some $a \in
\mathbb{R}$, $\phi \in \mathrm{GL}(n)$ and $x_0 \in \mathbb{R}^n$,
\[f(x)=a \left (1+|\phi(x-x_0)|^{p/(p-1)} \right )^{-(p-1)/(q-p)}.  \]
\end{corollary}

\end{document}